\documentclass[10pt]{article}
\usepackage{amsmath, amsthm, amssymb, amsfonts, mathrsfs, mathtools}
\usepackage[dvips]{graphicx}
\newtheorem{theorem}{Theorem}[section]
\newtheorem{cor}[theorem]{Corollary}

\newtheorem{lem}[theorem]{Lemma}
\newtheorem{prop}[theorem]{Proposition}
\newtheorem{definition}{Definition}
\newtheorem{ex}{Example}[section]

\def \Zl {{\mathbb Z}}
\def \Nl {{\mathbb N}}
\def \Rl {{\mathbb R}}

\title{ Perfect State Transfer on NEPS of the path $P_{3}$}

\author{ Hiranmoy Pal\\
Department of Mathematics\\
Indian Institute of Technology Guwahati\\
Guwahati, India - 781039\\
Email: hiranmoy@iitg.ernet.in\\
\\
Bikash Bhattacharjya\\
Department of Mathematics\\
Indian Institute of Technology Guwahati\\
Guwahati, India - 781039\\
Email: b.bikash@iitg.ernet.in
}

\begin{document}
\maketitle

\vspace{-0.3in}

\begin{center}{Abstract}\end{center}

Perfect state transfer is significant in quantum communication networks. There are very few graphs having this property. So, it is useful to find some new graphs having perfect state transfer. A good way to construct new graphs is by forming NEPS. It is known that the graph $P_{3}$ exhibits perfect state transfer and so we investigate some NEPS of the path $P_{3}$. A sufficient condition is found for a NEPS of $P_{3}$ to have perfect state transfer. Using these NEPS, some other graphs are constructed having perfect state transfer. We also prove that for every $n\in \Nl\setminus \left\lbrace 1\right\rbrace$ and any odd positive integer $k< n$, there is a basis $\Omega$ such that $NEPS(P_{3},\ldots, P_{3};\Omega)$ is connected and exhibits perfect state transfer.

 
\section{Introduction}
Perfect state transfer is highly desirable in quantum-communication networks modelled by a graph with adjacency matrix as the Hamiltonian of the system. The property of perfect state transfer on quantum networks was originally introduced by S. Bose \cite{bose}. The main goal is to find graphs having perfect state transfer. Christandl et al. \cite{chr1,chr2} shown that Cartesian products of the path $P_{2}$ and $P_{3}$ exhibit perfect state transfer. Again Bernasconi et al. \cite{ber} generalized the result of Christandl et al. \cite{chr1,chr2} for the graph $P_{2}$ and shown that the cubelike Cayley graphs $X\left(\Zl^{n}_{2}, \Omega\right)$, which are actually $NEPS\left(P_{2},\ldots P_{2}, \Omega\right)$, admit perfect state transfer whenever $\sum\limits_{w\in\Omega}w\neq \mathbf{0}$ in $\Zl_{2}^{n}$. So the natural question is to find whether $NEPS\left(P_{3},\ldots P_{3}, \Omega\right)$ admits perfect state transfer or not. This problem of finding perfect state transfer in NEPS of the path $P_{3}$ was asked by Dragan Stevanovi\'c in \cite{stev}. We show that some restrictions on the basis set $\Omega$ yields perfect state transfer on $NEPS\left(P_{3},\ldots P_{3}, \Omega\right)$. Although the NEPS may not always be connected for every basis $\Omega$, but with an additional condition we find that there are some NEPS of $P_{3}$ which are connected and exhibits perfect state transfer. This generalizes the result of Christandl et al. \cite{chr1,chr2} for the graph $P_{3}$.

\section{Preliminaries}
\subsection{Perfect state transfer on graphs}
Throughout the paper we only consider simple graphs. The transition matrix $H_{A}(t)$ for a graph $G$ with adjacency matrix $A$ is defined by $H_{A}(t):=\exp{(-itA)}=\sum\limits_{k\geq 0}(-i)^{k}A^{k}\frac{t^{k}}{k!}$. 
The graph $G$ is said to admit perfect state transfer from a vertex $u$ to another vertex $v$ at time $\tau\in\Rl$ if $\mid e_{u}^{T}H_{A}(\tau)e_{v}\mid=1$, \emph{i.e.} if the $uv$-th entry of $H_{A}(\tau)$ has unit modulus. Note that $e_{u}$ denotes the unit vector of appropriate size with $u$-th entry $1$. If $u=v$ then we say that the graph is periodic at the vertex $u$. A graph is said to be periodic if it is periodic at each of its vertices. We illustrate this by the following example.\par
Consider the graph $P_{2}$, the path on two vertices. The path $P_{2}$ has the adjacency matrix
\begin{eqnarray*}
      A=\left(\begin{array}{rr}
      0 & 1 \\
      1 & 0 \end{array} \right)
      \end{eqnarray*}
and $A^{2}=I$, the identity matrix. The transition matrix is therefore $H_{A}(t)=\cos(t)I-i\sin(t)A$. This implies that $P_{2}$ has perfect state transfer at $t=\frac{\pi}{2}$ and the graph is periodic at $t=\pi$.\par
Spectral decomposition can be used to find the transition matrix efficiently. Let $\lambda_{1},\ldots, \lambda_{m}$ be the distinct eigenvalues of $A$ and the projections (idempotents) onto the corresponding eigenspaces be $E_{1},\ldots, E_{m}$. The spectral decomposition of $A$ is therefore $A=\sum\limits_{r=1}^{m}\lambda_{r}E_{r}$. Note that $E_{1}+\ldots+E_{m}=I$, $E_{r}^{2}=E_{r}$ and $E_{r}E_{s}=0$ for $r\neq s$, $1\leq r,s\leq m$. As the exponential function is defined on the eigenvalues of $A$, we have the transition matrix as follows $H_{A}(t)=\sum\limits_{r=1}^{m}\exp{(-it\lambda_{r})}E_{r}$. By Lagrange interpolation, there is a polynomial $p(x)$ of degree at most $m-1$ such that $p(\lambda_{r})=\exp{\left(-it\lambda_{r}\right)}$, $1\leq r\leq m$. This implies $p\left(A\right)=\sum\limits_{r=1}^{m}p\left(\lambda_{r}\right) E_{r}=\sum\limits_{r=1}^{m}\exp{(-it\lambda_{r})}E_{r}=H_{A}(t)$. So $H_{A}(t)$ is a polynomial in $A$ and hence $H_{A}(t)$ is symmetric. Also the transition matrix $H_{A}(t)$ is unitary as $H_{A}(t)\left(H_{A}(t)\right)^{*}=H_{A}(t)\overline{H_{A}(t)}=I$. In the following example the transition matrix has been calculated with the help of spectral decomposition.
\par Consider the graph $P_{3}$, a path of length two with three vertices $1,2$ and $3$, where both the vertices $1$ and $3$ are adjacent to the vertex $2$. The adjacency matrix of $P_{3}$ with respect to the usual ordering of vertices is given by
\begin{eqnarray*}
      A=\left(\begin{array}{rrr}
      0 & 1 & 0\\
      1 & 0 & 1\\
      0 & 1 & 0\end{array} \right).
      \end{eqnarray*}
The spectral decomposition of $A$ is therefore $A= -\sqrt{2}E_{1} + E_{2} +\sqrt{2}E_{3}$, where the idempotents are
\begin{eqnarray*}
       \begin{array}{lll}
     E_{1} = \frac{1}{4} \left(\begin{array}{rrr}
               1 & -\sqrt{2} & 1 \\
               -\sqrt{2} & 2 & -\sqrt{2} \\
               1 & -\sqrt{2} & 1 \end{array} \right), &
     E_{2} = \frac{1}{2} \left(\begin{array}{rrr}
               1 & 0 & -1 \\
               0 & 0 & 0 \\
               -1 & 0 & 1 \end{array} \right), &
     E_{3} = \frac{1}{4} \left(\begin{array}{rrr}
               1 & \sqrt{2} & 1 \\
               \sqrt{2} & 2 & \sqrt{2} \\
               1 & \sqrt{2} & 1 \end{array} \right). \end{array}
\end{eqnarray*}
So the transition matrix of $P_{3}$ is $H_{A}(t)= \exp{\left(-i (-\sqrt{2}) t\right)} E_{1} + E_{2} +\exp{\left(-i \sqrt{2} t\right)} E_{3}$. At $t=\frac{\pi}{\sqrt{2}}$ we have,
\begin{eqnarray}
 H_{A}\left(\frac{\pi}{\sqrt{2}}\right) = -E_{1} + E_{2} - E_{3}
                            & = & \left(\begin{array}{rrr}
                                0 & 0 &-1\\
                                0 &-1 & 0\\
                               -1 & 0 & 0 \end{array}\right).  \nonumber
\end{eqnarray} 
This implies that $P_{3}$ exhibits perfect state transfer from the vertex $1$ to $3$ at time $t = \frac{\pi}{\sqrt{2}}$ and it is periodic at the vertex $2$ at the same time.
\par Let $G$ be a graph with adjacency matrix $A$. The set of all automorphisms of $G$ is denoted by $Aut\left(G\right)$. If $G$ admits perfect state transfer from the vertex $u$ to $v$ at time $\tau$ then $H_{A}(\tau)e_{u}=\gamma e_{v}$, where $\gamma$ is a complex number of unit modulus. Let $f\in Aut\left(G\right)$ and $Q$ be the permutation matrix of $f$. Then $Q$ commutes with $A$ and hence $Q$ commutes with $H_{A}(\tau)$ as $H_{A}(\tau)$ is a polynomial in $A$. Therefore we have $H_{A}(\tau)Qe_{u}=QH_{A}(\tau)e_{u}=\gamma Qe_{v}$. Note that $Qe_{u}=e_{f(u)}$ and $Qe_{v}=e_{f(v)}$. Thus we have the following result deduced from the proof of the Lemma 4.1 in \cite{god2}.
\begin{lem}\label{f3.1}
Let $f$ be an automorphism of a graph $G$. If perfect state transfer occurs between the vertices $u$ and $v$ of $G$ then perfect state transfer occurs between the vertices $f(u)$ and $f(v)$.
\end{lem}
Using Lemma \ref{f3.1}, we find that if $f:G\rightarrow H$ is an isomorphism and if there is perfect state transfer from the vertex $u$ to $v$ in $G$ then there is perfect state transfer from the vertex $f(u)$ to $f(v)$ in $H$ as well.

\par More information regarding perfect state transfer and periodicity can be found in \cite{god1, stev}.
\subsection{Kronecker Product of graphs} The Kronecker product on two graphs $G_{1}$ and $G_{2}$ with vertex set $V_{1}$ and $V_{2}$ is the graph $G_{1}\times G_{2}$, whose vertex set is $V_{1}\times V_{2}$. Two vertices $(u_{1},v_{1})$ and $(u_{2},v_{2})$ are adjacent in $G_{1}\times G_{2}$ whenever $u_{1}$ is adjacent to $u_{2}$ in $G_{1}$ and $v_{1}$ is adjacent to $v_{2}$ in $G_{2}$. Let  $G_{1}$ and $G_{2}$ have the adjacency matrices $A$ and $B$, respectively, with respect to some ordering of vertices. Then $G_{1}\times G_{2}$ has the adjacency matrix $ C:= A\otimes B $ with respect to the dictionary ordering of vertices.
\par The next result enables us to find the transition matrix for Kronecker product of graphs when the transition matrix for one of the graphs is known. The result has already been proved in \cite{god1}. We include a different proof of the result.
\begin{prop}\cite{god1}\label{aa}
Let $G_{1}$ and $G_{2}$ be two graphs having adjacency matrices $A$ and $B$. If the spectral decomposition of $B=\sum\limits_{s=1}^{q}\mu_{s}F_{s}$ then $G_{1}\times G_{2}$ has the transition matrix $\sum\limits_{s=1}^{q} H_{A}(\mu_{s}t)\otimes F_{s}$, where $H_{A}(t)$ is the transition matrix for $G_{1}$.
\end{prop}
\begin{proof}
Let the spectral decomposition of $A$ be $A=\sum\limits_{r=1}^{p}\lambda_{r}E_{r}$. The spectral decomposition of $C= A\otimes B$ is therefore $C=\sum\limits_{r=1}^{p}\sum\limits_{s=1}^{q} \lambda_{r}\mu_{s}\left(E_{r}\otimes F_{s}\right)$. So the transition matrix for Kronecker product becomes
\begin{eqnarray*}
H_{C}(t) = \sum_{r=1}^{p}\sum_{s=1}^{q} \exp(-it\lambda_{r}\mu_{s})\left(E_{r}\otimes F_{s}\right)
& = &  \sum_{s=1}^{q}\left(\sum_{r=1}^{p}\exp(-it\lambda_{r}\mu_{s})E_{r}\right)\otimes F_{s} \\
& = & \sum_{s=1}^{q} H_{A}(\mu_{s}t)\otimes F_{s}.
\end{eqnarray*}
Hence the result.
\end{proof}
More information on perfect state transfer of Kronecker products can be found in \cite{ge, god1}.

\section{Definitions and basic results}
In this section, we provide some definitions and results that are useful in finding perfect state transfer on NEPS of $P_{3}$.
 \begin{definition}[Center of a matrix]
We define the center of a square matrix $A=\left(a_{i,j}\right)$ of odd order $n$ by $\mathscr{C}(A):=a_{\frac{n+1}{2},\frac{n+1}{2}}$.
\end{definition}

\begin{definition}
Let $A=\left(a_{i,j}\right)$ be a square matrix of odd order $n\geq3$. We define $\mathscr{M}_{3}(A)$ to be the $3\times 3$ principal sub-matrix of $A$ that lies in the rows indexed by $\left\lbrace \frac{n-1}{2}, \frac{n+1}{2}, \frac{n+3}{2}\right\rbrace$. 
\end{definition}
It is easy to see that both $\mathscr{C}$ and $\mathscr{M}_{3}$ are linear functions on the set of all matrices of odd order $n\geq 3$. \emph{i.e.} If $A$ and $B$ be two matrices of odd order $n\geq 3$ and $\alpha$ a scalar then $\mathscr{C}(\alpha A+B)=\alpha\mathscr{C}(A)+\mathscr{C}(B)$ and $\mathscr{M}_{3}(\alpha A+B)=\alpha\mathscr{M}_{3}(A)+\mathscr{M}_{3}(B)$.
\par In the following result we find the center $\mathscr{C}$ and $\mathscr{M}_{3}$ for Kronecker product of some matrices with appropriate sizes.
\begin{prop}\label{lb1}
Let $B_{1},\ldots,B_{n}$ be square matrices such that order of each matrix is an odd number greater than or equal to $3$. If $A= B_{1}\otimes \ldots \otimes B_{n}$ then $\mathscr{C}(A)=\prod\limits_{i=1}^{n}\mathscr{C}(B_{i})$ and $\mathscr{M}_{3}(A)=\left(\prod\limits_{i=1}^{n-1}\mathscr{C}(B_{i})\right)\mathscr{M}_{3}\left(B_{n}\right)$.
\end{prop}
\begin{proof}
For $n=2$ the result follows directly from the definition of Kronecker product. Let us assume $A'=B_{1}\otimes\ldots\otimes B_{k}$ and also let $C=B_{1}\otimes\ldots\otimes B_{k-1}$ so that $A'=C\otimes B_{k}$. This implies that $\mathscr{C}(A')=\mathscr{C}(C)\mathscr{C}(B_{k})=\left(\prod\limits_{i=1}^{k-1}\mathscr{C}(B_{i})\right)\mathscr{C}(B_{k})=\prod\limits_{i=1}^{k}\mathscr{C}(B_{i})$ and $\mathscr{M}_{3}(A')=\mathscr{C}(C)\mathscr{M}_{3}\left(B_{k}\right)=\left(\prod\limits_{i=1}^{k-1}\mathscr{C}(B_{i})\right)\mathscr{M}_{3}\left(B_{k}\right)$.
Hence the result follows by induction.
\end{proof}

Let $U_{1}, U_{2}$ be two unitary matrices of odd order $n\geq 3$. If both the matrices $\mathscr{M}_{3}(U_{1})$ and $\mathscr{M}_{3}(U_{2})$ are also unitary then
\begin{eqnarray}
U_{1} U_{2}& = &\left(\begin{array}{ccc}
      * & O & *\\
      O & \mathscr{M}_{3}(U_{1}) & O\\
      * & O & *\end{array} \right)
      \left(\begin{array}{ccc}
      * & O & *\\
      O & \mathscr{M}_{3}(U_{2}) & O\\
      * & O & *\end{array} \right) \nonumber\\
& = & \left(\begin{array}{ccc}
      * & O & *\\
      O & \mathscr{M}_{3}(U_{1})\mathscr{M}_{3}(U_{2}) & O\\
      * & O & *\end{array} \right). \nonumber
\end{eqnarray}
Notice that $\mathscr{M}_{3}(U_{1})\mathscr{M}_{3}(U_{2})$ is also a unitary matrix. Thus we have the following result which can be proved easily by induction.
\begin{prop}\label{f5.1}
If $U_{1},\ldots, U_{k}$ are unitary matrices of odd order $n\geq 3$ and $\mathscr{M}_{3}(U_{j})$ is unitary for each $1\leq j\leq k$ then $\mathscr{M}_{3}\left(\prod\limits_{j=1}^{k}U_{j}\right)=\prod\limits_{j=1}^{k}\mathscr{M}_{3}(U_{j})$.
\end{prop}

\section{Perfect state transfer on NEPS of $P_{3}$}

Let $\Omega$ be a set of $n$-tuples $\beta=\left(\beta_{1}, \ldots, \beta_{n}\right)$ of symbols $0$ and $1$, which does not contain the $n$-tuple $\left(0, \ldots, 0\right)$. The NEPS \cite{cev} on the graphs $G_{1}, \ldots, G_{n}$ with basis $\Omega$ is the graph $NEPS\left(G_{1}, \ldots, G_{n};\Omega\right)$ whose vertex set is $V\left(G_{1}\right)\times\ldots\times V\left(G_{n}\right)$. Two vertices $(x_{1}, \ldots, x_{n})$ and $(y_{1}, \ldots, y_{n})$ are adjacent in $NEPS\left(G_{1}, \ldots, G_{n}; \Omega\right)$ if and only if there is an $n$-tuple $\left(\beta_{1}, \ldots, \beta_{n}\right)$ in $\Omega$ such that $x_{i}=y_{i}$ in $G_{i}$ exactly when $\beta_{i}=0$ and $x_{i}$ is adjacent to $y_{i}$ in $G_{i}$ exactly when $\beta_{i}=1$. The graphs  $G_{1}, \ldots, G_{n}$ are called factors of the NEPS.
\par From now onwards, we consider the $n$-tuples of $\Omega$ as vectors in $\Zl_{2}^{n}$. Also we consider the basis $\Omega$ both as a set of vectors and as a matrix having these vectors as its rows. We denote the rank of the matrix $\Omega$ by $r(\Omega)$ over the field $\Zl_{2}$.
\par The following result shows that the transition matrix of a NEPS can be written as a product of the transition matrices of some of its spanning subgraphs. 

\begin{prop}\label{a1}
The graph $NEPS\left(G_{1}, \ldots,G_{n}; \Omega\right)$ has the transition matrix $H_{\Omega}(t)=\prod\limits_{\beta\in\Omega}H_{\beta}(t)$, where $H_{\beta}(t)$ is the transition matrix for the spanning subgraph $NEPS\left(G_{1}, \ldots,G_{n}; \left\lbrace\beta\right\rbrace\right)$, $\beta\in\Omega$.
\end{prop}
\begin{proof}
Let the graphs $G_{1}, \ldots,G_{n}$ have adjacency matrices $A_{1}, \ldots, A_{n}$, respectively, with respect to some ordering of vertices in each $G_{k}$ ($1\leq k\leq n$). Then the graph $NEPS\left(G_{1}, \ldots,G_{n}; \Omega\right)$ has the adjacency matrix $A_{\Omega} = \sum\limits_{\left(\beta_{1}, \ldots, \beta_{n}\right) \in \Omega} A_{1}^{\beta_{1}} \otimes \ldots\otimes A_{n}^{\beta_{n}}$ with respect to the dictionary ordering of vertices induced by the ordering of vertices in each of its factor $G_{k}$ ($1\leq k\leq n$). See \cite{cev} for details. For $\beta=\left(\beta_{1}, \ldots, \beta_{n}\right)\in\Omega$, let us consider $A_{\beta}=A_{1}^{\beta_{1}}\otimes \ldots\otimes A_{n}^{\beta_{n}}$. Here $A_{\beta}$ can be considered as the adjacency matrix for the spanning subgraph $NEPS\left(G_{1}, \ldots,G_{n};\left\lbrace\beta \right\rbrace \right)$. Notice that if $\beta,\delta \in\Omega$, then $A_{j}^{\beta_{j}}$ and $A_{j}^{\delta_{j}}$ are either $A_{j}$ or $I$ and so $A_{j}^{\beta_{j}}$, $A_{j}^{\delta_{j}}$ commutes. Using the property that $\left(S_{1}\otimes T_{1}\right)\left(S_{2}\otimes T_{2}\right)=\left(S_{1} S_{2}\right)\otimes\left(T_{1} T_{2}\right)$, we can see that if $S_{1}$ commutes with $S_{2}$ and $T_{1}$ commutes with $T_{2}$ then $S_{1}\otimes T_{1}$ commutes with $S_{2}\otimes T_{2}$. This implies $A_{\beta}A_{\delta}=A_{\delta}A_{\beta}$ for $\beta,\delta \in\Omega$. Therefore the transition matrix for $NEPS\left(G_{1}, \ldots,G_{n};\Omega\right)$ becomes
\begin{eqnarray*}
H_{\Omega}(t)
=\exp{\left(-it\sum\limits_{\beta\in\Omega}A_{\beta}\right)}
& = &\prod\limits_{\beta\in\Omega}\exp{\left(-itA_{\beta}\right)}, \text{ as $A_{\beta}A_{\delta}=A_{\delta}A_{\beta}$} \\
& = &\prod\limits_{\beta\in\Omega}H_{\beta}(t).
\end{eqnarray*}
Hence the result follows.
\end{proof}
So it is evident that if we can find the transition matrices $H_{\beta}(t)$ for each of the spanning subgraphs $NEPS\left(G_{1}, \ldots,G_{n};\left\lbrace\beta \right\rbrace \right)$ then we can find the transition matrix for the graph $NEPS\left(G_{1}, \ldots,G_{n}; \Omega\right)$ quite easily. 
\par From now onwards, we consider NEPS having factor graphs $P_{3}$ only.  Also, for each $\beta\in\Omega$, we denote $s(\beta)$ to be the number of non-zero entries in $\beta$. Let us denote $\tau_{n}=\frac{\pi}{\left(\sqrt{2}\right)^{n}}$, $n\in\Nl$ so that $\sqrt{2}\tau_{n+1}=\tau_{n}$ for every $n\in\Nl$. Consider the matrix
\vspace{-0.2 in}\begin{eqnarray*}
      P=\left(\begin{array}{ccc}
      0 & 0 & 1\\
      0 & 1 & 0\\
      1 & 0 & 0\end{array} \right).
      \end{eqnarray*}
Recall that the transition matrix for $P_{3}$ at $\tau_{1}=\frac{\pi}{\sqrt{2}}$ is $-E_{1} + E_{2} - E_{3}= -P$, where $E_{1}, E_{2}$ and $E_{3}$ are the idempotents corresponding to the adjacency matrix for $P_{3}$. In the following lemma we find the principal submatrix $\mathscr{M}_{3}\left(H_{\beta}(t)\right)$ at a fixed time, depending on $\beta$, where $H_{\beta}(t)$ is the transition matrix for the graph $NEPS\left(P_{3}, \ldots,P_{3}; \left\lbrace\beta\right\rbrace\right)$.

\begin{lem}\label{b1}
Let $H_{\beta}(t)$ be the transition matrix for $NEPS\left(P_{3}, \ldots,P_{3}; \left\lbrace\beta\right\rbrace\right)$. If $\beta=\left(\beta_{1},\ldots \beta_{n}\right)$ and $s(\beta)=k$ then $\mathscr{M}_{3}\left(H_{\beta}(\tau_{k})\right)$ is $-I$ or $-P$ according as $\beta_{n}$ is $0$ or $1$ with $H_{\beta}(-\tau_{k})=H_{\beta}(\tau_{k})$.
\end{lem}
\begin{proof}
We prove this by induction on $n$, the length of $\beta$.
For $n=1$ we have $s(\beta)=1$ as $\beta\neq \mathbf{0}$ by definition of NEPS. In this case the NEPS is the graph $P_{3}$ itself. Therefore $\mathscr{M}_{3}\left(H_{\beta}(\tau_{1})\right)=\mathscr{M}_{3}\left(-P\right)=-P$ and also $H_{\beta}(-\tau_{1})=H_{\beta}(\tau_{1})$ as $P^{-1}=P$.
\par Assume that the result is true for any $\beta$ of length $n=l$. Consider $\beta=\left(\beta_{1},\ldots,\beta_{l}, \beta_{l+1}\right)$ and let $\beta^{*}=\left(\beta_{1},\ldots \beta_{l}\right)$. If $s(\beta^{*})=k'$ then by our assumption $\mathscr{M}_{3}\left(H_{\beta^{*}}(\tau_{k'})\right)$ is $-I$ or $-P$ according as $\beta_{l}$ is equal to $0$ or $1$ with $H_{\beta^{*}}(-\tau_{k'})=H_{\beta^{*}}(\tau_{k'})$. Now we consider two cases according as $\beta_{l+1}$ is $1$ or $0$.\\
\textbf{Case I:} Let $\beta_{l+1}=1$ so that $s(\beta)=k'+1$. In this case the graph $NEPS\left(P_{3},\ldots, P_{3};\left\lbrace\beta\right\rbrace \right)$ is actually the Kronecker product $NEPS\left(P_{3}, P_{3},\ldots, P_{3};\left\lbrace\beta^{*}\right\rbrace \right)\times P_{3}$. Recall that the spectral decomposition of the adjacency matrix of $P_{3}$ is $A= -\sqrt{2}E_{1} + E_{2} +\sqrt{2}E_{3}$, where $E_{1}, E_{2}$ and $E_{3}$ are the idempotents. By using Proposition \ref{aa} we have
\begin{eqnarray}
H_{\beta}(\tau_{k'+1}) & = & H_{\beta^{*}}(-\sqrt{2}\tau_{k'+1})\otimes E_{1} + H_{\beta^{*}}(0)\otimes E_{2} + H_{\beta^{*}}(\sqrt{2}\tau_{k'+1})\otimes E_{3} \nonumber \\
& = & H_{\beta^{*}}(-\tau_{k'})\otimes E_{1} + I\otimes E_{2} + H_{\beta^{*}}(\tau_{k'})\otimes E_{3}, \text{ as $\sqrt{2}\tau_{k'+1}=\tau_{k'}$ }\nonumber \\
& = & H_{\beta^{*}}(\tau_{k'})\otimes \left(E_{1}+E_{3}\right) + I\otimes E_{2} \nonumber \\
& = & H_{\beta^{*}}(\tau_{k'})\otimes \left(E_{2}+P\right) + I\otimes E_{2}, \text{ as $E_{1}-E_{2}+E_{3}=P$}\nonumber \\
& = & \left(H_{\beta^{*}}(\tau_{k'})+ I\right)\otimes E_{2} + H_{\beta^{*}}(\tau_{k'})\otimes P.
\end{eqnarray}
This implies $\mathscr{M}_{3}\left(H_{\beta}(\tau_{k'+1})\right)=-P$ since $\mathscr{C}\left(H_{\beta^{*}}(\tau_{k'})\right)=-1$. Also $H_{\beta}(-\tau_{k'+1}) = H_{\beta}(\tau_{k'+1})$ because the negative sign can be absorbed in $(1)$.\\
\textbf{Case II:} Let $\beta_{l+1}=0$ so that $s(\beta)=k'$. In this case the adjacency matrix for $NEPS\left(P_{3},\ldots, P_{3};\left\lbrace\beta\right\rbrace \right)$ is $A_{\beta}=A_{\beta^{*}}\otimes I$, where $A_{\beta^{*}}$ is the adjacency matrix for $NEPS\left(P_{3},\ldots, P_{3};\left\lbrace\beta^{*}\right\rbrace \right)$ and $I$ is the identity matrix of order $3$. This implies that $H_{\beta}(\tau_{k'})= \exp{\left(-i\tau_{k'}\left(A_{\beta^{*}}\otimes I\right)\right)}= H_{\beta^{*}}(\tau_{k'})\otimes I$. Therefore we have $\mathscr{M}_{3}\left(H_{\beta}(\tau_{k'})\right)=\mathscr{C}\left(H_{\beta^{*}}(\tau_{k'})\right)I=-I$ with $H_{\beta}(-\tau_{k'})=H_{\beta}(\tau_{k'})$.
\end{proof}
In the following theorem we provide a sufficient condition for NEPS of $P_{3}$ to exhibit perfect state transfer. Later on we will generalize this theorem and at the end of this section we will provide the sufficient condition for an NEPS to be connected and allow perfect state transfer. Let us denote $U_{j}$ and $V_{j}$ to be the vertices of $NEPS\left(P_{3},\ldots,P_{3};\Omega\right)$ where $j$-th entry of $U_{j}$ and $V_{j}$ are $1$ and $3$, respectively, and the other remaining entries are $2$.
\begin{theorem}\label{f7}
Let $\Omega$ be a set of $n$-tuples such that for each $\beta=(\beta_{1},\beta_{2},\ldots,\beta_{n})\in\Omega$, $s(\beta)=k$ is fixed. Then the following holds for the graph $NEPS\left(P_{3},\ldots,P_{3};\Omega\right)$ at time $\tau_{k}$.
\begin{itemize}
\item[(I)] If $\sum\limits_{\beta\in\Omega}\beta_{j}\neq 0$ for some $j$ then the graph exhibits perfect state transfer between the pair of vertices $U_{j}$ and $V_{j}$.
\item[(II)] If $\sum\limits_{\beta\in\Omega}\beta_{j}= 0$ for some $j$ then the graph is periodic at the vertices $U_{j}$, $V_{j}$.
\item[(III)] The graph is periodic at the vertex $\left(2,\ldots,2\right)$.
\end{itemize}
\end{theorem}
\begin{proof}
Consider $\Omega_{j}=\left\lbrace \delta(\beta)=\left(\beta_{1},\ldots, \beta_{n}, \ldots, \beta_{j}\right): \beta\in\Omega \right\rbrace $ \emph{i.e}, each $\delta(\beta)\in\Omega_{j}$ is defined by interchanging the $j$-th entry and $n$-th entry of $\beta\in\Omega$. Now both the graphs $G:=NEPS\left(P_{3},\ldots,P_{3};\Omega\right)$ and\linebreak $G_{j}:=NEPS\left(P_{3},\ldots,P_{3};\Omega_{j}\right)$ have the common vertex set $V(P_{3})\times  \ldots\times V(P_{3})$. It is easy to check that the map $f:V(G)\rightarrow V(G_{j})$ defined by $f(v_{1},\ldots,v_{j},\ldots,v_{n})=(v_{1},\ldots,v_{n},\ldots,v_{j})$ is an isomorphism. Thus perfect state transfer occurs between $U_{j}$ and $V_{j}$ in $G$ iff perfect state transfer occurs between $U_{n}$ and $V_{n}$ in $G_{j}$. Similarly, the graph $G$ is periodic at the vertices $U_{j}$, $V_{j}$ iff $G_{j}$ is periodic at the vertices $U_{n}$, $V_{n}$. Note that if $\sum\limits_{\beta\in\Omega}\beta_{j}\neq 0$ in $G$ then $\sum\limits_{\delta(\beta)\in\Omega_{j}}\delta(\beta)_{n}\neq 0$ in $G_{j}$. Similarly, if $\sum\limits_{\beta\in\Omega}\beta_{j}= 0$ in $G$ then $\sum\limits_{\delta(\beta)\in\Omega_{j}}\delta(\beta)_{n}=0$ in $G_{j}$. Thus it is enough to prove the result for $j=n$. Let $r$ be the numbers of $\beta\in\Omega$ for which $\beta_{n}=1$. By Lemma \ref{b1}, $\mathscr{M}_{3}\left(H_{\beta}(\tau_{k})\right)$ is equal to $-I$ or $-P$ according as $\beta_{n}$ is $0$ or $1$. So $\mathscr{M}_{3}\left(H_{\beta}(\tau_{k})\right)$ is unitary for all $\beta\in\Omega$. Let $H_{\Omega}(t)$ be the transition matrix for $NEPS\left(P_{3},\ldots,P_{3};\Omega\right)$. By Proposition \ref{f5.1} and Proposition \ref{a1}, we have
\begin{eqnarray}
\mathscr{M}_{3}\left(H_{\Omega}(\tau_{k})\right)=\mathscr{M}_{3}\left(\prod\limits_{\beta\in\Omega}H_{\beta}(\tau_{k})\right)
&=&\prod\limits_{\beta\in\Omega}\mathscr{M}_{3}\left(H_{\beta}(\tau_{k})\right) \nonumber\\
&=&(-1)^{m}P^{r}, \text{ where $m=|\Omega|$.}
\end{eqnarray}
Note that first row of $\mathscr{M}_{3}\left(H_{\Omega}(\tau)\right)$, $\tau\in\Rl$ corresponds to the $\frac{3^{n}-1}{2}$-th row of $H_{\Omega}(\tau)$. Also there are $3^{n-1}$ vertices preceding to the vertex $\left(2,1,\ldots,1\right)$ in dictionary ordering. Thus the position of the row in $H_{\Omega}(\tau)$ corresponding to the vertex $\left(2,\ldots,2,1\right)$ is $3^{n-1}+3^{n-2}+\ldots +3^{1}+1=\frac{\left(3^{n}-1\right)}{2}$. Smilarly, the position of the column in $H_{\Omega}(\tau)$ corresponding to the vertex $\left(2,\ldots,2,3\right)$ is $\frac{\left(3^{n}-1\right)}{2}+2=\frac{\left(3^{n}+3\right)}{2}$ as $\left(2,\ldots,2,1\right)$, $\left(2,\ldots,2,2\right)$ and $\left(2,\ldots,2,3\right)$ are consecutive vertices in dictionary ordering. Thus $(1,3)$-th entry of $\mathscr{M}_{3}\left(H_{\Omega}(\tau)\right)$ is actually the $U_{n}V_{n}$-th entry of $H_{\Omega}(\tau)$.
\par $(I)$ If $\sum\limits_{\beta\in\Omega}\beta_{n}\neq 0$ then $r$ is odd. As $P^{2}=I$, using $(2)$ we have $(1,3)$-th entry of $\mathscr{M}_{3}\left(H_{\Omega}(\tau_{k})\right)$ is $(-1)^{m}$. Hence perfect state transfer occurs between the vertices $U_{n}$ and $V_{n}$ at time $\tau_{k}$.
\par $(II)$ If $\sum\limits_{\beta\in\Omega}\beta_{n}= 0$ then $r$ is even. Thus using $(2)$, we have $\mathscr{M}_{3}\left(H_{\Omega}(\tau_{k})\right)=(-1)^{m}I$ as $P^{2}=I$. Now $(1,1)$ and $(3,3)$-th entry of $\mathscr{M}_{3}\left(H_{\Omega}(\tau_{k})\right)$ correspond to $U_{n}U_{n}$, $V_{n}V_{n}$-th entry of $H_{\Omega}(\tau_{k})$. Thus the graph is periodic at the vertices $U_{n}$, $V_{n}$ at time $\tau_{k}$.
\par $(III)$ In both the cases $(I)$ and $(II)$, the $(2,2)$-th entry of $\mathscr{M}_{3}\left(H_{\Omega}(\tau_{k})\right)$ is $(-1)^{m}$. Hence the graph is periodic at the vertex $\left(2,\ldots,2\right)$ at time $\tau_{k}$.
\end{proof}
Let $J$ be the all $1$ matrix of order $n$ and let $I$ be the identity matrix of same order. The following example shows that the graph $NEPS\left(P_{3},\ldots,P_{3};J-I\right)$ exhibits perfect state transfer when $n$ is even. We will find later that this graph is indeed connected.
\begin{ex}
Let us consider the NEPS of $P_{3}$ with the basis $\Omega=J-I$. Observe that $s(\beta)=n-1$ for each $\beta\in\Omega$. Also we have $\sum\limits_{\beta\in\Omega}\beta=(1,\ldots,1)$ or $(0,\ldots,0)$ according as $n$ is even or odd. By using Theorem \ref{f7} we see that when $n$ is even, the graph admits perfect state transfer at time $\tau_{n-1}=\frac{\pi}{(\sqrt{2})^{n-1}}$ between $U_{j}$ and $V_{j}$ for each $1\leq j\leq n$. Again if $n$ is odd then the graph is periodic at time $\tau_{n-1}=\frac{\pi}{(\sqrt{2})^{n-1}}$ at the vertices $U_{j}$, $V_{j}$ for each $1\leq j\leq n$. In any case the graph is periodic at the vertex $(2,\ldots,2)$.
\end{ex}
Thus we can construct many graphs allowing perfect state transfer. The following result which is indeed proved in \cite{chr1, chr2}, can also be obtained as a corollary of Theorem \ref{f7}.
\begin{cor}\cite{chr1, chr2}\label{f7.1}
Cartesian product of $n$ copies of the path $P_{3}$ exhibits perfect state transfer at time $\frac{\pi}{\sqrt{2}}$ between the pair vertices $U_{j}$ and $V_{j}$ for $1\leq j\leq n$.  
\end{cor}
\begin{proof}
Cartesian product of $n$ copies of the path $P_{3}$ is actually NEPS with basis $\Omega=I$ \cite{cev}, where $I$ is the identity matrix of order $n$. It is easy to verify that for $\beta\in\Omega$, $s(\beta)=1$ and $\sum\limits_{\beta\in\Omega}\beta=(1,\ldots,1)$. By Theorem \ref{f7}, the graph admits perfect state transfer at time $\tau_{1}=\frac{\pi}{\sqrt{2}}$ between $U_{j}$ and $V_{j}$ for each $1\leq j\leq n$.
\end{proof}
We now extend Theorem \ref{f7} to construct more general NEPS of the path $P_{3}$ exhibiting perfect state transfer. In the next result we show that, for a certain type of NEPS of $P_{3}$, the transition matrix for a spanning subgraph is same as that of the whole graph at a fixed time $\tau$ (say). Now if the spanning subgraph admits perfect state transfer at time $\tau$ then so does the graph itself between the same pair of vertices. That is, in some sense, adding extra edges (following certain rules) do not disturb the property of perfect state transfer in some NEPS of $P_{3}$.  
\begin{theorem}\label{f8}
Let $\Omega$ be a set of $n$-tuples such that $s(\beta)$ is even (or odd) for all $\beta\in\Omega$. If $k=\min\limits_{\beta\in\Omega} s(\beta)$ and $\Omega^{*}=\left\lbrace\beta\in\Omega : s(\beta)=k \right\rbrace $ then $H_{\Omega}(\tau_{k})=H_{\Omega^{*}}(\tau_{k})$, where $H_{\Omega}(t)$ and $H_{\Omega^{*}}(t)$ are the transition matrices for the NEPS of $P_{3}$ corresponding to $\Omega$ and $\Omega^{*}$, respectively.
\end{theorem}
\begin{proof}
Assume that $\beta\in\Omega\setminus\Omega^{*}$ and $s(\beta)=k'$. So by our assumption $k'-k(\neq 0)$ is even and this implies that $\tau_{k}=\frac{\pi}{(\sqrt{2})^{k}}=(\sqrt{2})^{k'-k}\frac{\pi}{(\sqrt{2})^{k'}}=2^{\frac{k'-k}{2}}\tau_{k'}=2m_{k'}\tau_{k'}$ where $m_{k'}$ is a positive integer. By Lemma \ref{b1} we have $H_{\beta}(-\tau_{k'})=H_{\beta}(\tau_{k'})$ \emph{i.e.} $H_{\beta}(2\tau_{k'})=I$. Now for each $\beta\in\Omega\setminus\Omega^{*}$, $H_{\beta}(\tau_{k})=H_{\beta}(2m_{k'}\tau_{k'})=\left[H_{\beta}(2\tau_{k'})\right]^{m_{k'}}=I$. Thus we have by Proposition \ref{a1}
\begin{eqnarray*}
H_{\Omega}(\tau_{k})=\prod\limits_{\beta\in\Omega}H_{\beta}(\tau_{k})
& = &\prod\limits_{\beta\in\Omega^{*}}H_{\beta}(\tau_{k})\prod\limits_{\beta\in\Omega\setminus\Omega^{*}}H_{\beta}(\tau_{k})\\
& = &\prod\limits_{\beta\in\Omega^{*}}H_{\beta}(\tau_{k})\\
& = & H_{\Omega^{*}}(\tau_{k}).
\end{eqnarray*}
Hence the theorem follows.
\end{proof}
\textbf{Remark:} If $\sum\limits_{\beta\in\Omega^{*}}\beta\neq\mathbf{0}$ then by Theorem \ref{f7}, the graph $NEPS(P_{3},\ldots,P_{3};\Omega^{*})$ admits perfect state transfer. Therefore the graph $NEPS(P_{3},\ldots,P_{3};\Omega)$ also exhibit perfect state transfer between same pair of vertices as in $NEPS(P_{3},\ldots,P_{3};\Omega^{*})$.
\begin{ex}
Consider $G:=NEPS\left(P_{3},P_{3},P_{3};\Omega\right)$ where $\Omega=\lbrace\left(1,0,0\right),\left(0,1,0\right),\left(0,0,1\right),\left(1,1,1\right)\rbrace.$ Note that the number of non-zero entries in each tuple contained in $\Omega$ is odd and the minimum of those numbers is $1$. For $\Omega^{*}=\lbrace\left(1,0,0\right),\left(0,1,0\right),\left(0,0,1\right)\rbrace$, the graph $G':=NEPS\left(P_{3},P_{3},P_{3};\Omega^{*}\right)$ is actually the Cartesian product of $P_{3}$. It has already been shown in \cite{chr1, chr2} that Cartesian product of any copies of $P_{3}$ exhibits perfect state transfer at $\frac{\pi}{\sqrt{2}}$. Finally, Theorem \ref{f8} implies that $G$ admits perfect state transfer at time $\frac{\pi}{\sqrt{2}}$ between the same pair of vertices as in $G'$.
\end{ex}
We generalize this observation as a corollary which is a direct consequence of the Theorem \ref{f8}. From this corollary, we can construct many NEPS having perfect state transfer.
\begin{cor}
Let $\Omega$ be the set of $n$-tuples containing all the rows of the identity matrix of order $n$. Consider $\Omega'=\left\lbrace \beta\in\Zl_{2}^{n} :  s(\beta) \text{ is odd and } s(\beta)\neq 1\right\rbrace$. Then the graph $NEPS\left(P_{3},\ldots,P_{3};S\cup\Omega\right)$ exhibits perfect state transfer between the same pair of vertices as in $NEPS\left(P_{3},\ldots,P_{3};\Omega\right)$, for any $S\subseteq\Omega'$.
\end{cor}
The theory of perfect state transfer is considered only on connected graphs but the NEPS in Theorem \ref{f7} and Theorem \ref{f8} may not always be connected. The following results by Dragan Stevanovi\'c determines when a NEPS of $P_{3}$ is connected.

\begin{lem}\cite{stev1}\label{s1}
Let $B_{1},\ldots, B_{n}$ be connected bipartite graphs and let $C_{1},\ldots, C_{m}$ be connected non-bipartite graphs. Then $G=NEPS(B_{1},\ldots, B_{n},C_{1},\ldots, C_{m};\mathscr{B})$ has the same number of components as $G'=NEPS(B_{1},\ldots, B_{n};\mathscr{B}')$, where $\mathscr{B}'$ consists of the columns of $\mathscr{B}$ corresponding to the bipartite graphs. If $C$ is the vertex set of a component of $G'$ then $C\times V(C_{1})\times\ldots\times V(C_{m})$ is the vertex set of a component of $G$.
\end{lem}

\begin{theorem}\cite{stev1}\label{s2}
Let $B_{1},\ldots, B_{n}$ be connected bipartite graphs then $G=NEPS(B_{1},\ldots, B_{n}; \mathscr{B})$ is connected if and only if $r(\mathscr{B})=n$.
\end{theorem}

The graph $P_{3}$ is connected and also it is a bipartite graph. So if we impose this extra condition $r(\Omega)=n$ in Theorem \ref{f7} and Theorem \ref{f8} then the NEPS of $P_{3}$ corresponding to $\Omega$ is connected and exhibits perfect state transfer. The following theorem ensures that such graphs exist.

\begin{theorem}
For every $n\in \Nl\setminus \left\lbrace 1\right\rbrace$ and any odd positive integer $k< n$, there is a basis $\Omega$ such that $NEPS(P_{3},\ldots, P_{3};\Omega)$ is connected and exhibits perfect state transfer at time $\tau_{k}=\frac{\pi}{(\sqrt{2})^{k}}$.
\end{theorem}
\begin{proof}
Let $n\in \Nl\setminus \left\lbrace 1\right\rbrace$ and $k<n$ be an odd positive integer. It is enough to show that there is a matrix $\Omega$ of size $n$ over $\Zl_{2}$ such that each of its rows have exactly $k$ non-zero entries with $r(\Omega)=n$. Note that $r(\Omega)=n$ will imply $\sum\limits_{\beta\in\Omega}\beta\neq \mathbf{0}$. Then Theorem \ref{f7} will imply that the NEPS of $P_{3}$ corresponding to $\Omega$ has perfect state transfer at time $\tau_{k}=\frac{\pi}{(\sqrt{2})^{k}}$. 
\par We prove  this by induction on $n$. For the initial case $n=2$, the only possibility for $k$ is $1$. In this case 
\[\Omega= \left(\begin{array}{rr}
           1 & 0\\
           0 & 1 \end{array}\right)\] serves our purpose.
Assume that for $n=l$ and any odd positive integer $k<l$, there is such an $\Omega$. Now consider $n=l+1$. We consider the following cases.\\
\textbf{Case I.(l is even):} Let $k< l+1$ be odd so that $k<l$ as well. So, by induction hypothesis, there is a matrix $\Omega$ of size $l$ over $\Zl_{2}$ such that each of its rows have exactly $k$ non-zero entries with $r(\Omega)=l$. Now consider the block matrix\[\Omega'= \left(\begin{array}{rr}
           \Omega & \textbf{0}\\
           \delta & 1 \end{array}\right),\]where $\delta$ is any $l$-tuple with $k-1$ non-zero entries. Clearly each rows of $\Omega'$ contains exactly $k$ nonzero entries with $r(\Omega')=l+1$.\\
\textbf{Case II.(l is odd):} Let $k< l+1$ be odd so that either $k<l$ or $k=l$. If $k< l$ we use the method in the previous case to find such an $\Omega'$. For $k=l$, consider $\Omega'=J-I$, where $J$ is the all one matrix of order $l+1$ and $I$ is the identity matrix of same order. As $l+1$ is even, $\Omega'^{T}\Omega'=(J-I)^{T}(J-I)=(J-I)^{2}=(l+1)J-2J+I=I$ in $\Zl_{2}$ and this implies $r(\Omega')=l+1$ in $\Zl_{2}$. This completes the proof.
\end{proof}

In this section we have developed a sufficient condition for a NEPS of $P_{3}$ to be connected and exhibit perfect state transfer. We provide that condition as a theorem as follows.

\begin{theorem}[\textbf{Sufficient Condition}]\label{as} Let $\Omega$ be a set of $n$-tuples such that $r(\Omega)=n$ and also let $s(\beta)$ be even (or odd) for all $\beta\in\Omega$. Assume that $k=\min\limits_{\beta\in\Omega} s(\beta)$ and $\Omega^{*}=\left\lbrace\beta\in\Omega : s(\beta)=k \right\rbrace $. If $\sum\limits_{\beta\in\Omega^{*}}\beta\neq \mathbf{0}$ then $NEPS\left(P_{3},\ldots,P_{3};\Omega\right)$ allows perfect state transfer at time $\tau_{k}=\frac{\pi}{(\sqrt{2})^{k}}$.
\end{theorem}

\section{Some more graphs allowing perfect state transfer}
In the previous section we have constructed several, in fact infinitely many, NEPS of $P_{3}$ allowing perfect state transfer. Using those NEPS we construct some more graphs admitting perfect state transfer.
\begin{theorem}\label{f9}
Let the graph $NEPS\left(P_{3},\ldots,P_{3};\Omega\right)$ satisfies the conditions of Theorem \ref{as}. Also let $G$ be a graph and $r\in\Rl$ be such that $\frac{\lambda}{r}$ is an odd integer for every eigenvalue $\lambda$ of $G$. Then the graph $NEPS\left(P_{3},\ldots,P_{3};\Omega\right)\times G$ admits perfect state transfer at time $\frac{\tau_{k}}{r}$.
\end{theorem}
\begin{proof}
Let $NEPS\left(P_{3},\ldots,P_{3};\Omega\right)$ satisfies the conditions of Theorem \ref{as}. By Theorem \ref{f8} we have $H_{\Omega}(\tau_{k})=H_{\Omega^{*}}(\tau_{k})$, where $H_{\Omega}(t)$ and $H_{\Omega^{*}}(t)$ are the transition matrices for the NEPS of $P_{3}$ corresponding to $\Omega$ and $\Omega^{*}$, respectively. By Lemma \ref{b1}, we have $H_{\beta}(-\tau_{k})=H_{\beta}(\tau_{k})$ for each $\beta\in\Omega^{*}$. Therefore Proposition \ref{a1} implies that $H_{\Omega^{*}}(-\tau_{k})=H_{\Omega^{*}}(\tau_{k})$, \emph{i.e.}, $\left(H_{\Omega^{*}}(\tau_{k})\right)^{2}=I$. For every integer $m$,
\begin{eqnarray*}
H_{\Omega}((2m+1)\tau_{k})=\left(H_{\Omega}(\tau_{k})\right)^{2m+1}
=\left(H_{\Omega^{*}}(\tau_{k})\right)^{2m+1}
=H_{\Omega^{*}}(\tau_{k})
=H_{\Omega}(\tau_{k}).
\end{eqnarray*}
Let $B=\sum\limits_{s=1}^{l}\lambda_{s}F_{s}$ be the spectral decomposition for the adjacency matrix of $G$. By Proposition \ref{aa}, the transition matrix for $NEPS\left(P_{3},\ldots,P_{3};\Omega\right)\times G$ at time $\frac{\tau_{k}}{r}$ is \[H\left(\frac{\tau_{k}}{r}\right)=\sum\limits_{s=1}^{l} H_{\Omega}\left(\frac{\lambda_{s}}{r}\tau_{k}\right)\otimes F_{s}=\sum\limits_{s=1}^{l} H_{\Omega}(\tau_{k})\otimes F_{s}=H_{\Omega}(\tau_{k})\otimes \sum\limits_{s=1}^{l} F_{s}=H_{\Omega}(\tau_{k})\otimes I.\] Now $NEPS\left(P_{3},\ldots,P_{3};\Omega\right)$ allows perfect state transfer at time $\tau_{k}$ and hence $NEPS\left(P_{3},\ldots,P_{3};\Omega\right)\times G$ exhibits perfect state transfer at time $\frac{\tau_{k}}{r}$.
\end{proof}
Note that if the graph $G$ is non-bipartite and connected then $NEPS\left(P_{3},\ldots,P_{3};\Omega\right)\times G$ is also connected. This follows by Lemma \ref{s1} and Theorem \ref{s2}. We now consider the following example.
\begin{ex}
The complete graph $K_{m}$ has the eigenvalues $-1,\ldots, -1, m-1$. If $m$ is even then all the eigenvalues are odd. Also let $J$ be the all one matrix of order $n$ and $I$ the identity matrix of same order. If $n$ is even then the graph $NEPS\left(P_{3},\ldots,P_{3};J-I\right)$ is connected (as $r(J-I)=n$ in $\Zl_{2}$) and allows perfect state transfer at time $\frac{\pi}{(\sqrt{2})^{n-1}}$ (by Theorem \ref{as}). Consequently, by Theorem \ref{f9}, the graph $NEPS\left(P_{3},\ldots,P_{3};J-I\right)\times K_{m}$ allows perfect state transfer at time $\frac{\pi}{(\sqrt{2})^{n-1}}$. Note that the graph $NEPS\left(P_{3},\ldots,P_{3};J-I\right)\times K_{m}$ is also connected when $m\geq 3$.
\end{ex}

\section*{Conclusion}
Perfect state transfer on quantum networks modeled by NEPS of $P_{3}$ is considered where the adjacency matrix is assumed to be the hamiltonian of the quantum system. We have developed a method to find perfect state transfer in NEPS of $P_{3}$. Also we have shown that NEPS of $P_{3}$ with basis $\Omega$ exhibits perfect state transfer whenever the following holds:
\begin{itemize}
\item For every tuple in $\Omega$, the number of nonzero entries in $\beta$ \emph{i.e. $s(\beta)$} is either even or odd.
\item $\sum\limits_{\beta\in\Omega^{*}}\beta\neq\mathbf{0}$ in $\Zl_{2}$ where $\Omega^{*}=\left\lbrace\beta\in\Omega : s(\beta)=k \right\rbrace $ and  $k=\min\limits_{\beta\in\Omega} s(\beta)$. 
\end{itemize}
In particular, we have seen that the Cartesian product of $P_{3}$ exhibits perfect state transfer which has already been shown by Christandl et al. in \cite{chr1,chr2}. But not only the Cartesian product, there are several other, in fact infinitely many, NEPS of $P_{3}$ allowing perfect state transfer. We also found that for every $n\in \Nl\setminus \left\lbrace 1\right\rbrace$ and any odd positive integer $k< n$, there is a basis $\Omega$ such that $NEPS(P_{3},\ldots, P_{3};\Omega)$ is connected and exhibits perfect state transfer at time $\tau_{k}=\frac{\pi}{(\sqrt{2})^{k}}$. Finally, we have constructed several other graphs out of the NEPS of $P_{3}$ exhibiting perfect state transfer. It will be interesting to find whether the conditions of the Theorem \ref{as} are also necessary.

\section*{Acknowledgments} We sincerely thank the anonymous referee for useful comments on an earlier manuscript.

\end{document}